\documentclass[12pt]{amsart}

\usepackage{anysize}
\marginsize{3cm}{3cm}{3cm}{3cm}

 \def\res{\mathop{Res}}

 \usepackage{bbold}
 \usepackage{amsmath}
 \usepackage{tikz}
\usetikzlibrary{patterns}
\usepackage{pgfplots,caption}
\pgfplotsset{compat=1.12}
\usepgfplotslibrary{fillbetween}
\usepackage{cool}
\usepackage{color}
\renewcommand{\Re}{\operatorname{Re}}

\usepackage{amsthm}
\newtheorem{theorem}{Theorem}[section]

\newtheorem{lemma}[theorem]{Lemma}
\newtheorem{proposition}[theorem]{Proposition}
\theoremstyle{definition}

 \theoremstyle{remark}
\newtheorem*{remark}{Remark}
\usepackage{enumerate}
\usepackage[utf8]{inputenc}

\title{On a Multiple Dirichlet Series\\Associated to Binary Cubic Forms}
\author{Eun Hye Lee, Ramin Takloo-Bighash}

\begin{document}
\maketitle

\tableofcontents

\section{Introduction}

Let $X$ be the prehomogeneous vector space of binary cubic forms, and denote the standard element of $X$ by
$$f(x,y)=ax^3+bx^2y+cxy^2+dy^3.$$
Let $X^+$ be the space of binary cubic forms with $a>0$ and $b^2-3ac<0$. Geometrically, this would mean that the cubic function $y=ax^3+bx^2+cx+d$ is everywhere increasing.

Let $\Gamma_\infty=\left<\begin{pmatrix} 1&1\\0&1\end{pmatrix}\right>$ be the upper triangular unipotent elements of $\operatorname{SL}_2(\mathbb{Z})$. Then we have a left action of $\Gamma_\infty$ on $X^+$, given by $\gamma\circ f(x,y)=f((x,y)\circ \gamma^T)$.
There are four relative invariants of $X^+$ under the action of $\Gamma_\infty$, namely 
\begin{align*}
r_1(f)&=a, \\
r_2(f)&=b^2-3ac, \\
r_3(f)&=2b^3+27a^2d-9abc, \\
r_4(f)&=b^2c^2+18abcd-4ac^3-4b^3d-27a^2d^2.
\end{align*}

Note that the last one of these invariants $r_4$  is the discriminant of $f(x,y)$ and these four relative invariants are linearly independent. See Section 3.1 of \cite{MR1693411} for more information on the relative invariants. 

A coset representative for $\Gamma_\infty\setminus X^+$ is 
\begin{equation*}
\left\{(a,b,c,d)\in\mathbb{Z}^4: a>0, 0\le b\le 3a-1, b^2-3ac<0\right\}.
\end{equation*}

In order to understand the distribution of the relative invariants of binary cubic forms, we form the four variable multiple Dirichlet series
$$\sum_{f\in\Gamma_\infty \setminus X^+} \dfrac1{|r_1(f)|^{s_1}|r_2(f)|^{s_2}|r_3(f)|^{s_3}|r_4(f)|^{s_4}}.$$
At present, determining the fine analytic properties of the this zeta function appears out of reach. For this reason we consider a simpler two variable zeta function. Define a further equivalence relation on $\Gamma_\infty \backslash X^+$ by letting $f_1 \sim f_2$ if $r_1(f_1) = r_1(f_2)$ and $r_2(f_1) = r_2(f_2)$. We then set 
$$Z(s_1,s_2)= \sum_{{f\in(\Gamma_\infty \setminus X^+) / \sim  \atop r_2(f)  \text{ odd, sq. free}}} \dfrac1{|r_1(f)|^{s_1}|r_2(f)|^{s_2}}.$$

 \begin{theorem}\label{thm.main}The multiple Dirichlet series $Z(s_1,s_2)$ can be meromorphically continued to the whole of $\mathbb{C}^2$.\end{theorem}
 
We note that, as opposed to the original Shintani zeta functions in \cite{MR289428} where one considers over all equivalence classes of integral binary cubic forms, our zeta function involves only a portion of this set.  We make some comments about the possible poles of the zeta function in section 2.2.

\

We will present the proof of this theorem in the next section.  Many mathematicians have studied multiple Dirichlet series of the form
$$\sum_{m_1,\cdots,m_n\in \mathbb{N}} \dfrac{a(m_1,\cdots,m_n)}{m_1^{s_1}m_2^{s_2}\cdots m_n^{s_n}}.$$
See, e.g., \cite{MR2041614}, \cite{MR1343325}, \cite{MR2254662}, \cite{MR788407}, and \cite{MR1224051} to name a few publications, and our proof of the theorem follows the (by now standard) techniques used in these works.

\

M. Sato and Shintani \cite{MR0344230} first defined zeta functions associated to prehomogeneous vector spaces with single relative invariants, and proved some properties including analytic continuation and functional equations. Then F. Sat\={o} undertook the study of multiple Dirichlet series associated to prehomogeneous vector spaces. See \cite{MR676121}, \cite{MR662121} and \cite{MR695661} for more details.

\

Our work in this paper is directly motivated by the thesis of Li-Mei Lim \cite{Lim} who considered multiple Dirichlet series associated to the prehomogeneous vector space of ternary quadratic forms in an attempt to generalize the result of Chinta and Offen \cite{MR2885075} on the orthogonal period of a GL$_3$ Eisenstein series.
Namely, let $X_3$ be the space of ternary quadratic forms, and let $X_3^+$ be the space of positive definite ternary quadratic forms. Also, let $P$ be the minimal parabolic subgroup in $\operatorname{SL}_3(\mathbb{Z})$. The action of $P$ on $X_3^+$ has three relative invariants, namely, 
\begin{align*}r_1(T) & =  r(s^2 - 4pq) + (pu^2 - stu + qt^2)\\r_2(T)& =  s^2 - 4pq\\r_3(T) & =  p\end{align*}
for $$T=\begin{bmatrix} 2p & s&t\\s&2q&u\\t&u&2r\end{bmatrix} \in X_3^+(\mathbb{Z}).$$

Using these relative invariants, one can form a multiple Dirichlet series:
$$Z(s_1,s_2,s_3)=\sum_{T\in P\backslash X_3^+(\mathbb{Z})}\dfrac1{| r_1(T)|^{s_1}| r_2(T)|^{s_2}| r_3(T)|^{s_3}}.$$
Lim proved that 
the $k^{-s_1}$ coefficient in the series $Z(s_1,s_2,s_3)$ is both equal to a sum of special values of the minimal parabolic $\operatorname{GL}_3$ Eisenstein series, and a linear combination of double Dirichlet series which
 arise as Fourier coefficients of the Eisenstein series on a 
  double cover of $\operatorname{GL}_3$.



\

We would like to think of our work as only the beginning of an investigation into an interesting realm of problems. For example, it would be desirable to give interpretations of the multiple Dirichlet series we consider here in terms of canonical objects of the theory of automorphic forms, e.g., Eisenstein series on higher rank groups. 

\

The authors wish to thank Gautam Chinta for useful conversations, and for his hospitality during the first author's visit to CCNY in 2018. We also benefited from conversations with Li-Mei Lim. We also wish to thank the referee and Evan O'Dorney for pointing out many inaccuracies in earlier drafts of this paper.  The second author is partially supported by a Collaboration Grant from Simons Foundation.

\section{Proof of the main theorem}

Let $f (x) = ax^3 + bx^2 y + c xy^2 + d y^3 \in X^+$. Then $|r_1(f)|=a$ and $|r_2(f)|=3ac-b^2$. 
Then 
 \begin{equation}Z(s_1,s_2)=\sum_{\substack{a,b,c\in\mathbb{Z}\\a>0\\0\le b\le 3a-1\\b^2-3ac<0, \operatorname{sq.free,odd}}}\dfrac1{a^{s_1}(-b^2+3ac)^{s_2}}.\label{eq:multdiri}\end{equation}

The following is the broad outline of the proof of Theorem \ref{thm.main}:
\begin{enumerate}
\item We reduce the multiple Dirichlet series into a finite linear combination of standard type double zeta functions;
\item we find a fairly large domain of meromorphy;
\item we find a functional equation, and use the functional equation to extend the domain of meromorphy;  
\item and finally, we use a convexity argument to get the final result.
\end{enumerate}

\subsection{Reduction of the Multiple Dirichlet Series}
\begin{lemma} We formally have 
 $$\displaystyle Z(s_1,s_2)=\sum_{\substack{m,n\in\mathbb{N}\\n\operatorname{sq.free,odd}}} \dfrac{C(3m,-n)}{m^{s_1}n^{s_2}},$$ where $$C(m,n)=\#\{x \mod m : x^2\equiv n \mod m\}.$$ \end{lemma}
\begin{proof} Let $m=a$ and $n=-b^2+3ac$.  Write 
$$\displaystyle Z(s_1,s_2)=\sum_{\substack{m,n\in\mathbb{N}\\n\operatorname{sq.free,odd}}} \dfrac{a(m, n)}{m^{s_1}n^{s_2}},$$
for integers $a(m, n)$.  Then, since $b^2=3ac-n\equiv -n\mod 3m$, the coefficient $a(m, n)$  is $$\#\{x\mod 3m: x^2\equiv -n \mod 3m\}.$$
\end{proof}
The following proposition shows the explicit formulae for $C(m,n)$.

\begin{proposition} The following properties hold. \begin{enumerate}[i)]\item For any fixed n, $C(m,n)$ is a weekly multiplicative function in $m$. In particular, $C(1,n)=1$ for all $n$.
\item If any prime $p\ne 2$ and $p\nmid n$, then $C(p^{\alpha},n)=1+\left(\dfrac{n}{p}\right)$ for $\alpha>0$.
\item If $p=2$ and $n$ is odd, then for $\alpha>0$, $$C(2^{\alpha},n)=\begin{cases} 1 & \alpha=1,\\ 2 & \alpha=2, n\equiv 1 \mod 4, \\ 4 & \alpha\ge 3, n\equiv 1\mod 8, \\0 & \mbox{otherwise.}\end{cases}$$
\item If $n=p^r n_0$ with $p\nmid n_0$, then for $\alpha>0$, $$C(p^{\alpha},p^rn_0)=\begin{cases} p^{\left\lfloor\frac{\alpha}{2}\right\rfloor} & r\ge\alpha,  \\ p^{\frac{r}2}C(p^{\alpha-r},n_0) & r<\alpha, r \mbox{ even},\\ 0 & \mbox{otherwise.}\end{cases}$$ \end{enumerate}
\end{proposition}
\begin{proof}\begin{enumerate}[i)]\item This is true by the Chinese Remainder Theorem. In particular, if $m=p_1^{\alpha_1}p_2^{\alpha_2}\cdots p_k^{\alpha_k}$ for all $p_i$'s prime, then $$C(m,n)=\prod_i C(p_i^{\alpha_i},n),$$ hence, we only have to analyze $C(p^{\alpha},n)$ for any prime $p$, $\alpha$, and $n$.
\item By Hensel's Lemma, if $p$ is an odd prime and $p\nmid n$, $x^2\equiv n\mod p^{\alpha}$ is solvable if and only if $x^2\equiv n\mod p$ is solvable and in such cases, $x^2\equiv n\mod p^{\alpha}$ has exactly two solutions.
\item The cases of $\alpha=1$ and $\alpha =2$ are trivial. 

\

Now, let's look at the cases of $\alpha\ge3$.\newline
\begin{remark} The equation $x^2\equiv1\mod 2^{\alpha}$ for $\alpha\ge3$ has exactly four solutions, which are $1$, $-1$, $1+2^{\alpha-1}$, and $-1+2^{\alpha-1}$.\end{remark}

For $x_1$ and $x_2$ odd, if $x_1^2\equiv n\mod 2^\alpha$ and $x_2^2\equiv n\mod 2^\alpha$, then $x_1^2\equiv x_2^2\mod 2^\alpha$, hence $(x_1x_2^{-1})^2\equiv 1 \mod 2^\alpha$. By the above remark, $x_1x_2^{-1}$ will be one of the four. Hence, if there is a solution to $x^2\equiv n\mod 2^{\alpha}$, then there will be exactly four of them.

Now, let's show that $x^2\equiv n\mod 2^{\alpha}$ is solvable if and only if $x^2\equiv n \mod 8$ is solvable, which is when $n\equiv 1\mod8$.

Recall the generalized version of Hensel's Lemma. For $p$ a prime and $f(x)\in\mathbb{Z}[x]$, if $p^{2N+1}|f(a)$ and $p^N|f'(a)$ but $p^{N+1}\nmid f'(a)$, then for all $M>N$, there is unique $x_M \mod p^M$ such that $f(x_M)\equiv0\mod p^M$, and $x_M\equiv a\mod p^{N+1}$.

Take $f(x)=x^2-n$ for an odd $n$, $p=2$, $N=1$. Then if $2^3|a^2-n$, $2|2a$, and $2^2\nmid 2a$, i.e. $a$ odd and $a^2-n\equiv0\mod 8$, then for all $M>1$, there exists unique $x_M$ such that $x_M^2-n\equiv0\mod 2^M$. That is to say, $x^2\equiv n\mod 2^\alpha$ is solvable if and only if $x^2\equiv n\mod 8$ is solvable, i.e., if  $n\equiv 1\mod 8$.

\item We first show that $C(p^\alpha,0)=p^{\left\lfloor\frac\alpha2\right\rfloor}$. Write $x$ in base $p$. Then 
\begin{align*}x&=a_kp^k+a_{k+1}p^{k+1}+\cdots +a_{\alpha-1}p^{\alpha-1}\\&=p^k\left(a_k+a_{k+1}p+\cdots+a_{\alpha-1}p^{\alpha-k-1}\right)\\
x^2&=p^{2k}(*)\equiv 0\mod p^{\alpha}
\end{align*}
So we need $2k\ge\alpha$, i.e. $k\ge\left\lceil\frac\alpha2\right\rceil$. Then $x=a_{\left\lceil\frac\alpha2\right\rceil}p^{\left\lceil\frac\alpha2\right\rceil}+\cdots+a_{\alpha-1}p^{\alpha-1}$ and any of these will work. Since the number of choices for each $a_i$ is $p$, $C(p^{\alpha},0)=p^{(\alpha-1)-\left\lceil\frac\alpha2\right\rceil+1}=p^{\alpha-\left\lceil\frac\alpha2\right\rceil}=p^{\left\lfloor\frac\alpha2\right\rfloor}$.

Now, let $x^2\equiv p^rn_0\mod p^{\alpha}$ with $p\nmid n_0$. If $r\ge \alpha$, then $x^2\equiv 0 \mod p^\alpha$, so $C(p^\alpha,p^rn_0)=p^{\left\lceil\frac\alpha2\right\rceil}$. For $r<\alpha$, $x^2\equiv p^rn_0\mod p^\alpha$ shows that $r$ is even.  Write $r=2\ell$, then $x=p^\ell y$ for some $y$. Then
$x^2=p^{2\ell}y^2=p^ry^2\equiv p^rn_0\mod p^\alpha$, hence \begin{equation}y^2\equiv n_0\mod p^{\alpha-r}.\label{eq:y}\end{equation} We know the number of solutions to \eqref{eq:y} for $y$ is $C(p^{\alpha-r},n_0)$. Now we have to lift $y$'s up to mod $p^\alpha$. The number of solutions of $y$ for $y^2\equiv n_0 \mod p^\alpha$ is $p^rC(p^{\alpha-r},n_0)$. Since $x=p^\ell y$ and multiplication by $p^\ell$ is a $p^\ell$-to-1 function mod $p^\alpha$, there are $p^\ell C(p^{\alpha-r},n_0)$ many solutions of $x$ to $x^2\equiv p^rn_0 \mod p^\alpha$. 
\end{enumerate}
\end{proof}
Next, $Z(s_1,s_2)$ can be rewritten as $$\displaystyle Z(s_1,s_2)=\sum_{\substack{n=1\\n \operatorname{sq.free, odd}}}^\infty \dfrac1{n^{s_2}}\sum_{m=1}^\infty\dfrac{C(3m,-n)}{m^{s_1}}.$$ Let $\displaystyle Z_n(s)=\sum_{m=1}^\infty\dfrac{C(3m,-n)}{m^s}$.
\begin{proposition} $\displaystyle Z_n(s)=\prod_{p \mbox{ }\operatorname{prime}} Z_{n,p}(s)$, for $$Z_{n,p}(s)=\begin{cases}\displaystyle \sum_{k=0}^\infty\dfrac{C(p^k,-n)}{p^{ks}} & p\ne3, \\ \displaystyle \sum_{k=0}^\infty\dfrac{C(3^{k+1},-n)}{3^{ks}} & p=3. \end{cases}$$\end{proposition}
\begin{proof} Trivial due to the multiplicativity of $C(m,n)$ for fixed $n$.\end{proof}
It is worth noting that the above general results are presented for the reference. We will use only the cases in which $n$ is square-free and odd.
\subsubsection{Local Euler Factors}
We explicitly compute the local Euler factors via local computations for various $p$'s and $n$'s.

\

\paragraph*{CASE 1: $p\nmid 6n$}

\begin{align*} Z_{n,p}(s)&=\sum_{k=0}^\infty\dfrac{C(p^k,-n)}{p^{ks}}=1+\sum_{k=1}^\infty\dfrac{1+\left(\frac{-n}p\right)}{p^{ks}}=1+\dfrac{p^{-s}\left(1+\left(\frac{-n}p\right)\right)}{1-p^{-s}}\\
&=\dfrac{1+p^{-s}\left(\frac{-n}p\right)}{1-p^{-s}}\cdot\dfrac{1-p^{-s}\left(\frac{-n}p\right)}{1-p^{-s}\left(\frac{-n}p\right)}=\dfrac{1-p^{-2s}}{1-p^{-s}}\cdot\dfrac1{1-p^{-s}\left(\frac{-n}p\right)}.\end{align*}
\paragraph*{CASE 2: $p=2$, $n$ odd}

\

\begin{enumerate}[i)]
\item $n\equiv 1\mod 8$, i.e., $-n\equiv 7\mod 8$:
$$Z_{n,2}(s)=\sum_{k=0}^\infty\dfrac{C(2^k,-n)}{2^{ks}}=1+\dfrac1{2^s}+\sum_{k=2}^\infty\dfrac{C(2^k,-n)}{2^{ks}}=1+\dfrac1{2^s}=\dfrac{1-2^{-2s}}{1-2^{-s}}.$$
\item $n\equiv 3\mod 8$, i.e., $-n\equiv 5\mod 8$:
$$Z_{n,2}(s)=\sum_{k=0}^\infty\dfrac{C(2^k,-n)}{2^{ks}}=1+\dfrac1{2^s}+\dfrac2{2^{2s}}+\sum_{k=3}^\infty\dfrac{C(2^k,-n)}{2^{ks}}=1+\dfrac1{2^s}+\dfrac2{2^{2s}}.$$
\item $n\equiv 5\mod 8$, i.e., $-n\equiv 3\mod 8$:
$$Z_{n,2}(s)=\sum_{k=0}^\infty\dfrac{C(2^k,-n)}{2^{ks}}=1+\dfrac1{2^s}+\sum_{k=2}^\infty\dfrac{C(2^k,-n)}{2^{ks}}=1+\dfrac1{2^s}=\dfrac{1-2^{-2s}}{1-2^{-s}}.$$
\item $n\equiv 7 \mod 8$, i.e., $-n\equiv 1\mod 8$:
\begin{align*} Z_{n,2}(s)&=\sum_{k=0}^\infty \dfrac{C(2^k,-n)}{2^{ks}}=1+\dfrac1{2^s}+\dfrac2{2^{2s}}+\sum_{k=3}^\infty \dfrac4{2^{ks}}\\
&=1+2^{-s}+2\cdot2^{-2s}+4\cdot\dfrac{2^{-3s}}{1-2^{-s}}=\dfrac{1+2^{-2s}+2\cdot2^{-3s}}{1-2^{-s}}\\
&=\dfrac{1-2^{-2s}}{1-2^{-s}}\cdot\dfrac{2\cdot2^{-2s}-2^{-s}+1}{1-2^{-s}}.\end{align*}
\end{enumerate}
\paragraph*{CASE 3: $p\ne 3$, $p|n$, i.e., $n=p^rn_0$}
\begin{enumerate}[i)]
\item $r$ is odd:
\begin{align*} Z_{n,p}(s)&=\sum_{k=0}^\infty\dfrac{C(p^k,-p^rn_0)}{p^{ks}}\\
&=\sum_{k=0}^r\dfrac{C(p^k,-p^rn_0)}{p^{ks}}+\sum_{k=r+1}^\infty\dfrac{C(p^k,-p^rn_0)}{p^{ks}}\\
&=\sum_{k=0}^r\dfrac{p^{\left\lfloor\frac{k}2\right\rfloor}}{p^{ks}}\\
&=1+\dfrac{p^0}{p^s}+\dfrac{p^1}{p^{2s}}+\dfrac{p^1}{p^{3s}}+\dfrac{p^2}{p^{4s}}+\dfrac{p^2}{p^{5s}}+\cdots+\dfrac{p^{\frac{r-1}2}}{p^{(r-1)s}}+\dfrac{p^{\frac{r-1}2}}{p^{rs}}\\
&=\left(1+\dfrac1{p^2}\right)\left(1+\dfrac{p}{p^{2s}}+\dfrac{p^2}{p^{4s}}+\cdots+\dfrac{p^{\frac{r-1}2}}{p^{(r-1)s}}\right)\\
&=\left(1+\dfrac1{p^s}\right)\dfrac{1-p^{(1-2s)\left(\frac{r+1}2\right)}}{1-p^{1-2s}}\\
&=\dfrac{1-p^{-2s}}{1-p^{-s}}\dfrac{1-p^{(1-2s)\left(\frac{r+1}2\right)}}{1-p^{1-2s}}.\end{align*}
In particular, if $r=1$, then $Z_{n,p}(s)=\dfrac{1-p^{-2s}}{1-p^{-s}}$.
\item $r$ is even:
\begin{align*}Z_{n,p}(s)&=\sum_{k=1}^\infty\dfrac{C(p^k,-p^rn_0)}{p^{ks}}\\
&=\sum_{k=0}^r\dfrac{p^{\left\lfloor\frac{k}2\right\rfloor}}{p^{ks}}+\sum_{k=r+1}^\infty\dfrac{p^{\frac{r}2}C(p^{k-r},-n_0)}{p^{ks}}\\
&=\sum_{k=0}^{r-1}\dfrac{p^{\left\lfloor\frac{k}2\right\rfloor}}{p^{ks}}+\dfrac{p^{\frac{r}2}}{p^{rs}}+\dfrac{p^{\frac{r}2}}{p^{rs}}\sum_{k=r+1}^\infty\dfrac{C(p^{k-r},-n_0)}{p^{(k-r)s}}\\
&=\sum_{k=0}^{r-1}\dfrac{p^{\left\lfloor\frac{k}2\right\rfloor}}{p^{ks}}+\dfrac{p^{\frac{r}2}}{p^{rs}}\sum_{k=0}^\infty\dfrac{C(p^k,-n_0)}{p^{ks}}\\
&=\dfrac{1-p^{-2s}}{1-p^{-s}}\cdot \dfrac{1-p^{(1-2s)\frac{r}2}}{1-p^{1-2s}}+\dfrac{p^{\frac{r}2}}{p^{rs}} Z_{n_0,p}(s)\\
&=\dfrac{1-p^{-2s}}{1-p^{-s}}\cdot\dfrac{1-p^{(1-2s)\frac{r}2}}{1-p^{1-2s}}+\dfrac{p^{\frac{r}2}}{p^{rs}}\cdot\dfrac{1-p^{-2s}}{1-p^{-s}}\cdot\dfrac1{1-p^{-s}\left(\frac{-n_0}p\right)}\\
&=\dfrac{1-p^{-2s}}{1-p^{-s}}\left(\dfrac{1-p^{(1-2s)\frac{r}2}}{1-p^{1-2s}}+\dfrac{p^{\frac{r}2}}{p^{rs}}\dfrac1{1-p^{-s}\left(\frac{-n_0}{p}\right)}\right)\\
&=\dfrac{1-p^{-2s}}{1-p^{-s}}\cdot\dfrac{\left(1-p^{(1-2s)\frac{r}2}\right)p^{rs}\left(1-p^{-s}\left(\frac{-n_0}p\right)\right)+\left(1-p^{1-2s}\right)p^{\frac{r}2}}{(1-p^{1-2s})p^{rs}\left(1-p^{-s}\left(\frac{-n_0}p\right)\right)}\\
&=\dfrac{1-p^{-2s}}{1-p^{-s}}\cdot\dfrac{\left(p^{rs}-p^{\frac{r}2}\right)\left(1-p^{-s}\left(\frac{-n_0}p\right)\right)+p^{\frac{r}2}-p^{1-2s+\frac{r}2}}{(1-p^{1-2s})p^{rs}\left(1-p^{-s}\left(\frac{-n_0}p\right)\right)}\\
&=\dfrac{1-p^{-2s}}{1-p^{-s}}\cdot\dfrac{p^{rs}-p^{rs-s}\left(\frac{-n_0}p\right)+p^{\frac{r}2-s}\left(\frac{-n_0}p\right)-p^{1-2s+\frac{r}2}}{(1-p^{1-2s})p^{rs}\left(1-p^{-s}\left(\frac{-n_0}p\right)\right)}\\
&=\dfrac{1-p^{-2s}}{1-p^{-s}}\cdot\dfrac{p^{rs}\left(1-p^{-s}\left(\frac{-n_0}p\right)\right)+p^{\frac{r}2-s}\left(\frac{-n_0}p\right)\left(1-p^{1-s}\left(\frac{-n_0}p\right)\right)}{(1-p^{1-2s})p^{rs}\left(1-p^{-s}\left(\frac{-n_0}p\right)\right)}.\end{align*}
\end{enumerate}
\paragraph*{CASE 4: $p=3$}
\begin{enumerate}[i)]
\item $3\nmid n$:
\begin{align*}
Z_{n,3}(s)=\sum_{k=0}^\infty\dfrac{C(3^{k+1},-n)}{3^{ks}}=\sum_{k=0}^\infty \dfrac{1+\left(\frac{-n}3\right)}{3^{ks}}=\left(1+\left(\dfrac{-n}3\right)\right)\dfrac1{1-3^{-s}}.
\end{align*}
In particular, if $-n\equiv 2\mod 3$, i.e. $n\equiv1\mod 3$, then $Z_{n,3}(s)=0$.
\item $n=3^rn_0$, $3\nmid n_0$, $r$ odd:
\begin{align*}
Z_{n,3}(s)&=\sum_{k=0}^\infty \dfrac{C(3^{k+1},-3^rn_0)}{3^{ks}}\\
&=\sum_{k=0}^{r-1}\dfrac{C(3^{k+1},-3^rn_0)}{3^{ks}}+\sum_{k=r}^\infty \dfrac{C(3^{k+1},-3^rn_0)}{3^{ks}}\\
&=\sum_{k=0}^{r-1}\dfrac{3^{\left\lfloor\frac{k+1}2\right\rfloor}}{3^{ks}}\\
&=1+\dfrac3{3^s}+\dfrac3{3^{2s}}+\dfrac{3^2}{3^{3s}}+\dfrac{3^2}{3^{4s}}+\cdots+\dfrac{3^{\frac{r-1}2}}{3^{(r-2)s}}+\dfrac{3^{\frac{r-1}2}}{3^{(r-1)s}}\\
&=1+\left(\dfrac3{3^s}+\dfrac3{3^{2s}}\right)\left(1+\dfrac3{3^{2s}}+\dfrac{3^2}{3^{4s}}+\cdots+\dfrac{3^{\frac{r-3}2}}{3^{(r-3)s}}\right)\\
&=1+\left(\dfrac3{3^s}+\dfrac3{3^{2s}}\right)\left(\dfrac{1-3^{(1-2s)\left(\frac{r-1}2\right)}}{1-3^{1-2s}}\right).
\end{align*}
In particular, if $r=1$, then $Z_{n,3}(s)=1$.
\item $n=3^rn_0$, $3\nmid n_0$, $r$ even:
\begin{align*}
Z_{n,3}(s)&=\sum_{k=0}^\infty \dfrac{C(3^{k+1},-3^rn_0)}{3^{ks}}\\
&=\sum_{k=0}^{r-1}\dfrac{3^{\left\lfloor\frac{k+1}2\right\rfloor}}{3^{ks}}+\sum_{k=r}^\infty \dfrac{3^{\frac{r}2}C(3^{k+1-r},-n_0)}{3^{ks}}\\
&=\left(1+\dfrac3{3^s}+\dfrac3{3^{2s}}+\dfrac{3^2}{3^{3s}}+\cdots+\dfrac{3^{\frac{r-2}2}}{3^{(r-2)s}}+\dfrac{3^{\frac{r}2}}{3^{(r-1)s}}\right)+\dfrac{3^{\frac{r}2}}{3^{rs}}\sum_{k=r}^\infty\dfrac{C(3^{k+1-r},-n_0)}{3^{(k-r)s}}\\
&=\left(1+\dfrac3{3^s}\right)\left(1+\dfrac3{3^{2s}}+\dfrac{3^2}{3^{4s}}+\cdots+\dfrac{3^{\frac{r}2-1}}{3^{(k-2)s}}\right)+\dfrac{3^{\frac{r}2}}{3^{rs}}\sum_{\ell=0}^\infty \dfrac{C(3^{\ell+1},-n_0)}{3^{\ell s}}\\
&=\left(1+3^{1-s}\right)\dfrac{1-3^{(1-2s)\frac{r}2}}{1-3^{1-2s}}+3^{\frac{r}2-rs}Z_{n_0,3}(s).
\end{align*}
\end{enumerate}

We collect these results in the following proposition: 

\begin{proposition} \label{thm.diri} Define a Dirichlet character $\eta_{-n}(p)$ by setting $\eta_{-n}(p)=\left(\dfrac{-n}p\right)$, and denote the Dirichlet $L$-function associated to $\eta_{-n}$ by $L(s,\eta_{-n})$. \begin{enumerate}[i)] \item If $n\equiv 1 \mod 3$, then $Z_n(s)=0$. \item If $n=q_1^{r_1}q_2^{r_2}\cdots q_k^{r_k}$ for $q_i\ne 2,3$ primes, $r_i \ge 1$ odd, then\begin{enumerate}\item $n\equiv 5\mod 12$:
$$Z_n(s)=\dfrac{\zeta(s)}{\zeta(2s)}L(s,\eta_{-n})(1-2^{-s})  \dfrac{2}{1-3^{-2s}}\prod_{q_i}\dfrac{1-q_i^{(1-2s)\left(\frac{r_i+1}2\right)}}{1-q_i^{1-2s}}.$$
\item $n\equiv 23\mod 24$:
$$Z_n(s)=\dfrac{\zeta(s)}{\zeta(2s)}L(s,\eta_{-n})(2\cdot2^{-2s}-2^{-s}+1)\dfrac{2}{1-3^{-2s}}\prod_{q_i}\dfrac{1-q_i^{(1-2s)\left(\frac{r_i+1}2\right)}}{1-q_i^{1-2s}}.$$
\item $n\equiv 11\mod 24$:
$$Z_n(s)=\dfrac{\zeta(s)}{\zeta(2s)}L(s,\eta_{-n})\dfrac{1-2^{-s}}{1+2^{-s}}(1+2^{-s}+2\cdot2^{-2s})\dfrac{2}{1-3^{-2s}}\prod_{q_i}\dfrac{1-q_i^{(1-2s)\left(\frac{r_i+1}2\right)}}{1-q_i^{1-2s}}.$$
\end{enumerate}
\item Define 
$$A_j(s)=\begin{cases}0 & j \equiv 1\mod 3, \\\dfrac{1}{1+3^{-s}}&j\equiv0\mod6,\\\dfrac{2}{1-3^{-2s}}& j\equiv2\mod6,\\(1-2^{-s})\dfrac{2}{1-3^{-2s}}& j\equiv5\mod 12,\\(1-2^{-s})\dfrac{1}{1+3^{-s}}&j\equiv 9\mod 12,\\\dfrac{1-2^{-s}}{1+2^{-s}}(1+2^{-s}+2\cdot2^{-2s})\dfrac{1}{1+3^{-s}}&j\equiv 3\mod 24,\\\dfrac{1-2^{-s}}{1+2^{-s}}(1+2^{-s}+2\cdot2^{-2s})\dfrac{2}{1-3^{-2s}}&j\equiv11\mod24,\\(1-2^{-s}+2\cdot2^{-2s})\dfrac{1}{1+3^{-s}}&j\equiv15\mod24,\\(1-2^{-s}+2\cdot2^{-2s})\dfrac{2}{1-3^{-2s}}&j\equiv 23\mod24.\end{cases}
$$
Then, if $n$ is square-free, 
$$Z_n(s)=\frac{\zeta(s)}{\zeta(2s)}L(s,\eta_{-n})A_n(s).$$
\end{enumerate} 
\end{proposition}
So from the above computations, letting $n$ be square-free, we conclude that 
$$
Z(s_1,s_2)=\frac{\zeta(s_1)}{\zeta(2s_1)}\sum_{n \, \operatorname{sq.free, odd}}a_n(s_1) \frac{L_{2, 3}\left(\left(\frac{-n}{\cdot} \right), s_1\right)}{n^{s_2}}.
$$
Here, $L_{2, 3}$ is a Dirichlet $L$-function with the Euler factors at $2, 3$ removed, and $a_n(s_1)$ is a function defined to correct those removed Euler factors, which depends on the residue of $n$ modulo 24.

\

Let us assume that $n$ is coprime to $6$. Let $\delta_j(n)$ to be equal to $1$ if $n \equiv j \mod 24$, otherwise equal to $0$. Then 
$$
\delta_j(n) = \frac{1}{\phi(24)} \sum_{\chi: (\mathbb Z/24\mathbb Z)^\times \to \mathbb C^\times}\chi(j)^{-1} \chi(n).
$$
Then if $\gcd(n, 24) = 1$, we have 
\begin{align*}
a_n(s)& = \sum_{ {j = 1 \atop \gcd(j, 24)=1}}^{24} A_j(s) \delta_j(n) \\
&=  \frac{1}{\phi(24)}\sum_{ {j = 1 \atop \gcd(j, 24)=1}}^{24} A_j(s)  \sum_{\chi: (\mathbb Z/24\mathbb Z)^\times \to \mathbb C^\times}\chi(j)^{-1} \chi(n) \\
&= \frac{1}{8} \sum_{ {j = 1 \atop \gcd(j, 24)=1}}^{24} \sum_{\chi: (\mathbb Z/24\mathbb Z)^\times \to \mathbb C^\times}\chi(j)^{-1} A_j(s) \chi(n).
\end{align*}
Hence
\begin{align*}
Z(s_1, s_2) &= \frac{\zeta(s_1)}{\zeta(2s_1)}\sum_{n \, \operatorname{sq.free, odd}}a_n(s_1) \frac{L_{2,3} \left(\left(\frac{-n}{\cdot} \right), s_1\right)}{n^{s_2}}\\
&= \frac{1}{8}\frac{\zeta(s_1)}{\zeta(2s_1)}\sum_{n \, \operatorname{sq.free, odd}} \frac{L_{2,3} \left(\left(\frac{-n}{\cdot} \right), s_1\right)}{n^{s_2}}\sum_{ {j = 1 \atop \gcd(j, 24)=1}}^{24} \sum_{\chi: (\mathbb Z/24\mathbb Z)^\times \to \mathbb C^\times}\chi(j)^{-1} A_j(s_1) \chi(n) \\
&= \frac{1}{8}\frac{\zeta(s_1)}{\zeta(2s_1)}\sum_{ {j = 1 \atop \gcd(j, 24)=1}}^{24} \sum_{\chi: (\mathbb Z/24\mathbb Z)^\times \to \mathbb C^\times}\chi(j)^{-1} A_j(s_1) \sum_{n \, \operatorname{sq.free, odd}} \frac{\chi(n) \cdot L_{2,3} \left(\left(\frac{-n}{\cdot} \right), s_1\right)}{n^{s_2}} . 
\end{align*}
This shows that $Z(s_1, s_2)$ is a finite linear combination of standard type multiple Dirichlet series of the form 
$$
\tilde Z_\chi(s_1,s_2):=\sum_{n \, \operatorname{sq.free, odd}} \frac{\chi(n) \cdot L_{2,3} \left(\left(\frac{-n}{\cdot} \right), s_1\right)}{n^{s_2}}  
$$
and this will be analyzed in the following section.

\subsection{Domain of Meromorphy}
In this section, we adapt the methods of Diaconu, Goldfeld and Hoffstein from \cite{MR2041614}.

\

Let $w,s_1,s_2,\cdots,s_m\in\mathbb{C}$ with $\Re(w)>1, \Re(s_i)>1$ for $i=1,\cdots , m$. Consider the absolutely convergent multiple Dirichlet series
$$Z(w, s_1,s_2,\cdots,s_m)=\sum_d\dfrac{\chi(d)L(s_1,\chi_d)\cdots L(s_m,\chi_d)}{|d|^w}$$
where $\chi(d)$ is some Dirichlet character and $$\chi_d(n)=\begin{cases} \left(\dfrac{n}d\right) & n \text{ odd}\\ (-1)^{\frac{d^2-1}8} & n=2\\ \operatorname{sign}(d) & n=-1.\end{cases}$$
To see where the poles and residues of $Z(w,s_1,\cdots,s_m)$ are, let us define an adjusted multiple Dirichlet series:
$$Z_{\nu}^{\pm}(w,s_1,\cdots,s_m)=\sum_{\substack{\pm d>0\\d\equiv \nu \mod 4\\d:\operatorname{sq.free}\\{{3\nmid d}}}}\dfrac{\chi(d)L(s_1,\chi_d)\cdots L(s_m,\chi_d)}{|d|^w}.$$
We set $$Z^{\pm}(w,s_1,\cdots,s_m)=Z_1^{\pm}(w,s_1,\cdots,s_m)+Z_3^{\pm}(w,s_1,\cdots,s_m).$$
\begin{proposition}\label{prop:mero} For $\sigma>0$, $Z^{\pm}$ can be continued meromorphically to the domain  $\Re(w)>1+m\left(\frac12+\sigma\right)$, and $\Re(s_i)>-\sigma$ for $i=1,\cdots,m$.  In this region, the only poles are at $s_i=1$ for $i=1,\cdots,m$.\end{proposition}
Before proving the statement, it is worth noting that we do not find the exact poles, but the above Proposition \ref{prop:mero} can be used to precisely determine the poles of $\tilde Z_\chi(s_1,s_2)$. Since $Z(s_1,s_2)$ is a linear combination of $\tilde Z_\chi (s_1,s_2)$, the set of poles of $Z(s_1,s_2)$ is a subset of that of $Z^*(s_1,s_2)$. 
\begin{proof}We can rearrange the series and see the following:\newline$\displaystyle Z_1^\pm(w,s_1,\cdots,s_m)=\sum_{\substack{\pm d>0\\d\equiv1\mod4\\d:\operatorname{sq.free}\\3\nmid d}}\dfrac{\chi(d)L(s_1,\chi_d)\cdots L(s_m,\chi_d)}{|d|^w}\\ \phantom{Z_1^\pm(s_1,\cdots,s_m,w)}=\sum_{\substack{\pm d>0\\d\equiv 1 \mod 4\\d:\operatorname{sq.free}\\3\nmid d}}\dfrac{\chi(d)}{|d|^w}\prod_{j=1}^m\sum_{n_j=1}^\infty\dfrac{\chi_d(n_j)}{n_j^{s_j}}\\\phantom{Z_1^\pm(s_1,\cdots,s_m,w)}=\sum_{n_1,\cdots,n_m}\dfrac1{n_1^{s_1}\cdots n_m^{s_m}}\sum_{\substack{\pm d>0\\d\equiv 1\mod 4\\d:\operatorname{sq.free}\\3\nmid d}}\dfrac{\chi(d)\chi_d(n_1n_2\cdots n_m)}{|d|^w}.$\newline For any fixed $m$-tuple of positive integers, $\{n_1,\cdots,n_m\}$,  we can write $$n_1\cdots n_m=2^cnN^2M^2$$ with the following properties:
\begin{itemize}\item $n$: square-free, \item If $p|N$, then $p|n$,\item $\gcd(M,n)=1$, \item $n$, $N$, $M$ odd.\end{itemize} Hence, the inner summation becomes:\newline
$\displaystyle \sum_{\substack{\pm d>0\\d\equiv 1\mod 4\\d:\operatorname{sq.free}\\3\nmid d}}\dfrac{\chi(d)\chi_d(n_1n_2\cdots n_m)}{|d|^w}=\sum_{\substack{\pm d>0\\d\equiv 1\mod 4\\d:\operatorname{sq.free}\\3\nmid d\\(d,M)=1}}\dfrac{\chi(d)\chi_d(2)^c\chi_d(n)}{|d|^w}
=\sum_{\substack{\pm d>0\\d\equiv 1\mod 4\\d\operatorname{sq.free}\\3\nmid d\\(d,M)=1}}\dfrac{\chi(d)\chi_2(d)^c\chi_n(d)}{|d|^w}
\\\phantom{\sum_{\substack{\pm d>0\\d\equiv 1\mod 4\\d:\operatorname{sq.free}\\3\nmid d}}\dfrac{\chi(d)\chi_d(n_1n_2\cdots n_m)}{|d|^w}}=\dfrac12\sum_{\substack{\pm d>0\\d\operatorname{sq.free}\\3\nmid d\\(d,M)=1\\2\nmid d}}\dfrac{\chi_2(d)^c\chi_n(d)\chi(d)}{|d|^w}+\dfrac12\sum_{\substack{\pm d>0\\d\operatorname{sq.free}\\3\nmid d\\(d,M)=1\\2\nmid d}}\dfrac{\chi_2(d)^c\chi_{-1}(d)\chi_n(d)\chi(d)}{|d|^w}.$
Before we continue we state and prove a lemma. Let $\psi$ be a Dirichlet character of conductor $n$. Define$$L_b(w,\psi)=\sum_{\substack{d>0\\d\operatorname{sq.free}\\(d,b)=1}}\dfrac{\psi(d)}{d^w}.$$

\begin{lemma}
For a Dirichlet character $\psi$ mod $n$, $L(2w,\psi^2)L_b(w,\psi)$ is holomorphic on $\Re(w)>0$ except when $\psi$ is trivial mod $n$, in which case, there is a unique pole at $w=1$ with residue $$\displaystyle \prod_{p|b}(1+p^{-1})^{-1}\prod_{p|n}(1-p^{-1})\prod_{p|n}(1-p^{-2}).$$
\end{lemma}
\begin{lemma}
For a Dirichlet character $\psi$ mod $n$, $L(2w,\psi^2)L_b(w,\psi)$ is holomorphic on $\Re(w)>0$ except when $\psi$ is trivial mod $n$, in which case, there is a unique pole at $w=1$ with residue $$\displaystyle \prod_{p|b}(1-p^{-1})^{-1}\prod_{p|n}(1-p^{-1})\prod_{p|n}(1-p^{-2}).$$
\end{lemma}
\begin{proof}
For a primitive $\psi_1 \mod \tilde{n}$, which is extended to $\psi\mod n$, we have 
\begin{align*} L_b(w,\psi)&=L_b(w,\psi_1)\prod_{p|n}(1+\psi_1(p)p^{-w})\\
&=\dfrac{L(w,\psi_1)}{L(2w,\psi_1^2)}\prod_{p|b}(1+\psi_1(p)p^{-w})^{-1}\prod_{p|n}(1+\psi_1(p)p^{-w}).\end{align*}
Since $\displaystyle L(2w,\psi_1^2)=L(2w,\psi^2)\prod_{p|n}\dfrac1{1-\psi_1^2(p)p^{-2w}}$, 
\begin{align*}L(2w,\psi^2)L_b(w,\psi)=L(w,\psi_1)\prod_{p|b}(1+\psi_1(p)p^{-w})^{-1}\prod_{p|n}(1+\psi_1(p)p^{-w})
\\ \times\prod_{p|n}(1-\psi_1(p)^2p^{-2w})\end{align*}
has a pole at $w=1$ if and only if $\psi_1$ is trivial. Since $\psi_1$ is primitive, in order to have a pole, $\tilde{n}=1$. Hence, the residue of $L(2w,\psi^2)L_b(w,\psi)$ at $w=1$ is $$\displaystyle \prod_{p|b}(1+p^{-1})^{-1}\prod_{p|n}(1-p^{-1})\prod_{p|n}(1-p^{-2}).$$
\end{proof}
\

We now return to the proof of Proposition \ref{prop:mero}.

\

For $$\displaystyle Z(w):=\sum_{\substack{\pm d>0\\d\equiv1\mod 4\\d\operatorname{sq.free}\\3\nmid d}}\dfrac{\chi(d)\chi_d(n_1\cdots n_m)}{|d|^w}=\dfrac12L_{6M}(w,\chi_2^c\chi_n\chi)+\dfrac12L_{6M}(w,\chi_2^c\chi_{-1}\chi_n\chi),$$ 
$L(2w,\chi^2)Z(w)$ is holomorphic on $\Re (w)>0$ and meromorphic with a pole at $w=1$ if and only if either $\chi_2^c\chi_n\chi$ or $\chi_2^c\chi_{-1}\chi_n\chi$ is trivial. In particular, if $\chi$ is not quadratic, then there is no pole for $L(2w,\chi^2)Z(w)$. Note that $\chi_2^c\chi_n\chi$ and $\chi_2^c\chi_{-1}\chi_n\chi$ cannot be both trivial.
\end{proof}

\begin{proposition}\label{thm.res}
For a Dirichlet character $\chi$,  we have 
$$\res_{s_1=\frac12}\res_{s_2=1}L(2,\chi^2)\tilde Z_\chi(s_1,s_2)=\dfrac13\prod_{p|\operatorname{cond}\chi}(1-p^{-1})\prod_{p\ne2,3}\dfrac{1-\frac{\chi^2(p)}{p^2}+\frac{\chi^2(p)}{p^{2s_1+2}}-\frac1{p^{2s_1+1}}}{1-\frac{\chi^2(p)}{p^2}},$$
provided that the primitive part of $\chi$ is trivial. Otherwise, there is no pole.
\end{proposition}

In order to prove above proposition, we will need the following lemma:
\begin{lemma}\label{mini}
For a character $\psi(n) \operatorname{mod} N$, $\displaystyle \sum_{\substack{n\operatorname{sq.free}\\(n,m)=1}}\dfrac{\psi(n)}{n^w}=\dfrac{L(w,\psi)}{L_m(2w,\psi^2)}$ for $m|\operatorname{cond}\psi$.
\end{lemma}
\begin{proof} We have 
\begin{align*}
\sum_{\substack{n\operatorname{sq.free}\\(n,m)=1}}\dfrac{\psi(n)}{n^w} & =\prod_{p\nmid m}\left(1+\dfrac{\psi(p)}{p^w}\right) \\ & =\prod_{p\nmid m}\dfrac{(1-\psi(p)p^{-w})^{-1}}{(1-\psi(p)^2p^{-2w})^{-1}} \\
& =\dfrac{L(w,\psi)}{L_m(2w,\psi^2)}.
\end{align*}
\end{proof}
\begin{proof}[Proof of Proposition \ref{thm.res}]
Using Lemma \ref{mini}, we get the follwoing:
\begin{align*}
\tilde Z_\chi(s_1,s_2)&=\sum_{n \operatorname{sq.free}} \dfrac{\chi(n)L_{2,3}\left(s_1,\left(\frac{-n}{\cdot}\right)\right)}{n^{s_2}}\\
&=\sum_{n\operatorname{sq.free,odd}} \dfrac{\chi(n)}{n^{s_2}}\sum_{\substack{m=1\\\gcd(6,m)=1}}^\infty \dfrac{\left(\frac{-n}m\right)}{m^{s_1}}\\
&=\sum_{\substack{m=1\\\gcd(6,m)=1}}^\infty \dfrac{\left(\frac{-1}m\right)}{m^{s_1}}\sum_{\substack{n\operatorname{sq.free}\\\gcd(n,m)=1}}\dfrac{\chi(n)\left(\frac{n}m\right)}{n^{s_2}}\\
&=\sum_{\substack{m=1\\\gcd(6,m)=1}}^\infty \dfrac{\left(\frac{-1}m\right)}{m^{s_1}}\dfrac{L\left(s_2,\chi(\cdot)\left(\frac{\cdot}m\right)\right)}{L_m(2s_2,\chi^2)}.
\end{align*}
Hence, we have the following identity:
$$L(2s_2,\chi^2)\tilde Z_\chi(s_1,s_2)=\sum_{\substack{m=1\\(6,m)=1}}^\infty\dfrac{\left(\frac{-1}m\right)}{m^{s_1}}\dfrac{L(2s_2,\chi^2)}{L_m(2s_2,\chi^2)}L\left(s_2,\chi(\cdot)\left(\frac{\cdot}m\right)\right).$$
Since $\displaystyle \dfrac{L(2s_2,\chi^2)}{L_m(2s_2,\chi^2)}=\prod_{p|m}\left(1-\dfrac{\chi^2(p)}{p^{2s_2}}\right)^{-1}$, we get
$$L(2s_2,\chi^2)\tilde Z_\chi(s_1,s_2)=\sum_{\substack{m=1\\\gcd(6,m)=1}}^\infty \dfrac{\left(\frac{-1}m\right)}{m^{s_1}}\prod_{p|m}\left(1-\dfrac{\chi^2(p)}{p^{2s_2}}\right)^{-1}L\left(s_2,\chi(\cdot)\left(\frac{\cdot}m\right)\right).$$
Now, if $m$ is a square, then $$\res_{s_2=1}L\left(s_2,\chi(\cdot)\left(\dfrac{\cdot}m\right)\right)=\res_{s_2=1}L\left(s_2,\chi(\cdot)\mathbb{1}_m\right)=\prod_{p|m}(1-p^{-1})\prod_{p|\operatorname{cond}\chi}(1-p^{-1})$$
and if $m$ is not a square, $$\res_{s_2=1}L\left(s_2,\chi(\cdot)\left(\dfrac{\cdot}m\right)\right)=0,$$ since $\operatorname{cond}\chi(\cdot)\left(\dfrac{\cdot}m\right)=\underbrace{\operatorname{cond}\chi(\cdot)}_{| 24}\underbrace{\operatorname{cond}\left(\dfrac{\cdot}m\right)}_{|m}\ne1$.
Therefore,
\begin{align*}
&\res_{s_2=1}L(2s_2,\chi^2)\tilde Z_\chi(s_1,s_2)\\
&=\sum_{\substack{m=1\\\gcd(6,m)=1\\m\operatorname{square}}}^\infty\dfrac{\displaystyle \prod_{p|m}\left(1-\dfrac{\chi^2(p)}{p^2}\right)^{-1}\prod_{p|m}(1-p^{-1})\prod_{p|\operatorname{cond}\chi}(1-p^{-1})}{m^{s_1}}\\
&=\prod_{p|\operatorname{cond}\chi}(1-p^{-1})\sum_{\substack{m=1\\\gcd(6,m)=1\\m\operatorname{square}}}^\infty \dfrac{\displaystyle \prod_{p|m}\left(1-\frac{\chi^2(p)}{p^2}\right)^{-1}(1-p^{-1})}{m^{s_1}}\\
&=\prod_{p|\operatorname{cond}\chi}(1-p^{-1})\prod_{p\ne2,3}\left(1+\left(1-\dfrac{\chi^2(p)}{p^2}\right)^{-1}(1-p^{-1})\sum_{k=1}^\infty p^{-2ks_1}\right)\\
&=\prod_{p|\operatorname{cond}\chi}(1-p^{-1})\prod_{p\ne2,3}\left(1+(1-p^{-1})\left(1-\dfrac{\chi^2(p)}{p^2}\right)^{-1}\dfrac{p^{-2s_1}}{1-p^{-2s_1}}\right).
\end{align*}
Here, the inner product can be analyzed as follows:
\begin{align*}
\prod_{p\ne2,3}\left(1+\dfrac{1-p^{-1}}{1-\frac{\chi^2(p)}{p^2}}\frac{p^{-2s_1}}{1-p^{-2s_1}}\right)&=\prod_{p\ne2,3}\dfrac{\left(1-\frac{\chi^2(p)}{p^2}\right)(1-p^{-2s_1})+(1-p^{-1})p^{-2s_1}}{\left(1-\frac{\chi^2(p)}{p^2}\right)(1-p^{-2s_1})}\\
&=\prod_{p\ne2,3}\dfrac{1-\frac{\chi^2(p)}{p^2}+\frac{\chi^2(p)}{p^{2s_1+2}}-\frac1{p^{2s_1+1}}}{1-\frac{\chi^2(p)}{p^2}} \prod_{p\ne2,3}(1-p^{-2s_1})^{-1}.
\end{align*}
Recall that the product $\displaystyle \prod_n (1+a_n)$ converges absolutley if $\displaystyle\sum_n|a_n|<\infty$.
For $\Re s_1>0$, since
\begin{align*}
\sum_{p\ne2,3}\left|-\dfrac{\chi^2(p)}{p^2}+\dfrac{\chi^2(p)}{p^{2s_1+2}}-\dfrac1{p^{2s_1+1}}\right|\le\sum_{p\ne2,3}\dfrac1{p^2}+\sum_{p\ne 2,3} \dfrac1{p^{2\Re s_1+2}}+\sum_{p\ne2,3}\dfrac1{p^{2\Re s_1+1}}<\infty
\end{align*} and $$\displaystyle \sum_{p\ne2,3}\left|\dfrac{\chi^2(p)}{p^2}\right|\le \sum_{p\ne2,3}\dfrac1{p^2}<\infty,$$ we conclude
$$\prod_{p\ne2,3}\dfrac{1-\frac{\chi^2(p)}{p^2}+\frac{\chi^2(p)}{p^{2s_1+2}}-\frac1{p^{2s_1+1}}}{1-\frac{\chi^2(p)}{p^2}}$$
converges absolutely. Hence, the only pole of $\displaystyle \res_{s_2=1}L(2s_2,\chi^2)\tilde Z_\chi(s_1,s_2)$ comes from $$\displaystyle \prod_{p\ne2,3}(1-p^{-2s_1})^{-1}=(1-2^{-2s_1})(1-3^{-2s_1})\zeta(2s_1)$$ and the pole is at $s_1=\frac12$.

\

Finally, 
$$\res_{s_1=\frac12}\res_{s_2=1}L(2s_2,\chi^2)\tilde Z_\chi(s_1,s_2)=\dfrac13\prod_{p|\operatorname{cond}\chi}(1-p^{-1})\prod_{p\ne2,3}\dfrac{1-\frac{\chi^2(p)}{p^2}+\frac{\chi^2(p)}{p^{2s_1+2}}-\frac1{p^{2s_1+1}}}{1-\frac{\chi^2(p)}{p^2}}.$$
This concludes the proof of Proposition \ref{thm.res}.
\end{proof}

\subsection{Functional Equation}
Recall the fuctional equation of $L(s,\chi)$ for $\chi$ a primitive Dirichlet character mod $k$. Define 
\begin{equation}\Lambda(s,\chi):=\left(\dfrac\pi{k}\right)^{-\frac{s+a}2}\Gamma\left(\frac{s+a}2\right)L(s,\chi),\label{eq.func}\end{equation} where $$a=\begin{cases}0 & \mbox{ if } \chi(-1)=1\\1 & \mbox{ if } \chi(-1)=-1\end{cases}.$$ For $\displaystyle \tau(\chi)=\sum_{n=1}^\infty \chi(n)\exp(2\pi i n/k)$, we have the following functional equation:
$$\Lambda(1-s,\overline{\chi})=\dfrac{i^a k^{1/2}}{\tau(\chi)}\Lambda(s,\chi).$$
It is a well-known (see, e.g., Exercise 21 on page 45 of \cite{MR1625181}) that for $\chi$ a real primitive character modulo $k$ where $k>2$, 
\begin{equation}\label{eq.gauss}
\tau(\chi)=\begin{cases} \sqrt{k} & \text{ if } \chi(-1)=1\\ i\sqrt{k} & \text { if } \chi(-1)=-1\end{cases}.
\end{equation}
Now, 
\begin{align*}
L_{2,3}\left(\left(\dfrac{-n}{\cdot}\right),s\right)&=\sum_{\substack{(m,6)=1\\(m,n)=1}}\dfrac{\left(\frac{-n}{m}\right)}{m^s}=\sum_{m=1}^\infty\dfrac{\left(\frac{-n}{m}\right)\mathbb{1}_6(m)}{m^s}\\
&=\sum_{m=1}^\infty\dfrac{(-1)^{\frac{m-1}2}\left(\frac{n}m\right)\mathbb{1}_6(m)}{m^s}=\sum_{m=1}^\infty\dfrac{\chi_4(m)\left(\frac{n}m\right)\mathbb{1}_6(m)}{m^s}\\
&=\sum_{m=1}^\infty\dfrac{\chi_4(m)\left(\frac{n}m\right)\mathbb{1}_3(m)}{m^s}.
\end{align*}
In this case, if $m$ is even, then due to the presence of $\chi_4(m)$, the whole term disappears: take only odd $m$'s. Now, letting $n=2^cN$ for $N$ odd, 
$$\left(\dfrac{n}m\right)=\left(\dfrac{2^cN}{m}\right)=\left(\dfrac2m\right)^c\left(\dfrac{N}m\right)=(-1)^{\frac{m^2-1}8\cdot c}(-1)^{\frac{N-1}2\frac{m-1}2}\left(\dfrac{m}N\right).$$
This gives the conductor of $\left(\dfrac{n}m\right)=\begin{cases}8N & \mbox{ if } c\,\, \operatorname{odd}\\4N & \mbox{ if }c\,\,\operatorname{even}\end{cases}$, so that the conductor of $\chi_4(m)\left(\dfrac{n}m\right)\mathbb{1}_3(m)$ is $\begin{cases}3\cdot 8N& c\,\,\operatorname{odd}\\3\cdot 4N& c\,\,\operatorname{even}\end{cases}$.

\

Now we apply the above functional equation to the function 
$$\displaystyle \tilde Z_\chi (s_1,s_2)=\sum_{{n\operatorname{sq.free,odd}}}\dfrac{\chi(n)L_{2,3}\left(\left(\frac{-n}{\cdot}\right),s_1\right)}{n^{s_2}}
=\sum_{\substack{n\operatorname{sq.free,odd}}}\dfrac{\chi(n)L(\psi_n,s_1)}{n^{s_2}}$$ for $\psi_n$ primitive Dirichlet character mod $3\cdot4n$.

From \eqref{eq.func}, since $$\psi_n(-1)=\chi_4(-1)\cdot\left(\dfrac{n}{-1}\right)\cdot\mathbb{1}_3(-1)=-1\cdot \left(\dfrac{n}{4n-1}\right)\cdot 1=-1,$$ $a_n=1$, and we can define
\begin{equation}\Lambda(s_1,\psi_n):=\left(\dfrac\pi{12n}\right)^{-\frac{s_1+1}2}\Gamma\left(\dfrac{s_1+1}2\right)L(\psi_n,s_1).\label{def.Lambda}\end{equation}
Finally, the functional equation for $\Lambda$ is:
$$\Lambda(1-s_1, \overline{\psi_n})=\dfrac{i\cdot (12n)^{1/2}}{\tau(\psi_n)}\Lambda(s_1, \psi_n),$$
for $\tau(\psi_n)=\displaystyle \sum_{\ell=1}^{12n} \psi_n(\ell)\exp(2\pi i \ell/12n)$.

Since $\psi_n$ is real, $\overline{\psi_n}=\psi_n$. Hence,
$$\Lambda(1-s_1,\psi_n)=\dfrac{i\cdot (12n)^{1/2}}{\tau(\psi_n)}\Lambda(s_1,\psi_n).$$
Since $\tau(\psi_n)=i\cdot(12n)^{1/2}$ from \eqref{eq.gauss}, we have the following functional equation:
\begin{equation}\Lambda(1-s_1,\psi_n)=\Lambda(s_1,\psi_n).\label{eq.fct}\end{equation}

Using the definition of $\Lambda$ in \eqref{def.Lambda} and the functional equation in \eqref{eq.fct}, we get
$$\left(\dfrac{\pi}{12n}\right)^{-\frac{1-s_1+1}2}\Gamma\left(\dfrac{1-s_1+1}2\right)L(\psi_n,1-s_1)=\left(\dfrac{\pi}{12n}\right)^{-\frac{s_1+1}2}\Gamma\left(\dfrac{s_1+1}2\right)L(\psi_n,s_1).$$ Hence,
$$L(\psi_n,s_1)=\left(\dfrac{\pi}{12n}\right)^{s_1-\frac12}\dfrac{\Gamma\left(1-\frac{s_1}2\right)}{\Gamma\left(\frac{s_1+1}2\right)}L(\psi_n,1-s_1).$$
Going back to  $\tilde Z_\chi(s_1,s_2)$,
\begin{align*}
\tilde Z_\chi(s_1,s_2)=&\sum_{\substack{n\operatorname{sq.free,odd}}}\dfrac{\chi(n)L(\psi_n,s_1)}{n^{s_2}}\\
=&\sum_{\substack{n\operatorname{sq.free,odd}}}\dfrac{\chi(n)\left(\dfrac{\pi}{12n}\right)^{s_1-\frac12}\dfrac{\Gamma\left(1-\frac{s_1}2\right)}{\Gamma\left(\frac{s_1+1}2\right)}L(\psi_n,1-s_1)}{n^{s_2}}\\
=&\left(\dfrac{\pi}{12}\right)^{s_1-\frac12}\dfrac{\Gamma\left(1-\frac{s_1}2\right)}{\Gamma\left(\frac{s_1+1}2\right)}\tilde Z_\chi\left(1-s_1,s_1+s_2-\frac12\right).
\end{align*}

\subsection{Conclusion}
From the sections 4.2 and 4.3, we get the domain of meromorphy of 
\begin{itemize}
\item $\Re s_1>0, \Re s_2\gg 0$ and
\item $\Re s_1>\frac32, \Re s_1+\Re s_2>\frac32$.
\end{itemize}

\begin{center}

\begin{tikzpicture}[scale=1.2]
    \begin{axis}[
        xmin=-2.9, xmax=5.9,
        ymin=-2.9, ymax=5.9,
        axis lines=middle,
        yticklabel=\empty,
        xlabel=$\sigma_1$,ylabel=$\sigma_2$,label style =
               {at={(ticklabel cs:1.1)}}
      ]  
      \addplot[gray, samples=100, domain=0:5.7, name path=A] {3}; 
      \path[name path=xaxis] (\pgfkeysvalueof{/pgfplots/xmin}, 0) -- (\pgfkeysvalueof{/pgfplots/xmax},0);
      \path[name path=max] (\pgfkeysvalueof{/pgfplots/xmin},\pgfkeysvalueof{/pgfplots/ymax} ) -- (\pgfkeysvalueof{/pgfplots/xmax},\pgfkeysvalueof{/pgfplots/ymax});
      \path[name path=second] (1.5,0) -- (4.4,-2.9) -- (5.9,-2.9);
 \path[name path=hull] (0,3) -- (1.5,0);
      \addplot[pattern color=blue!40, pattern=north west lines] fill between[of=A and max, soft clip={domain=0:5.7}];
      \addplot[pattern color=red!40, pattern=north east lines] fill between[of=max and second, soft clip={domain=1.5:4.4}];
      \addplot[pattern color=red!40, pattern=north east lines] fill between[of=max and second, soft clip={domain=4.4:5.7}];
      \addplot[pattern color=yellow, pattern=horizontal lines] fill between[of=A and hull, soft clip={domain=0:1.5}];
     \addplot [gray, mark=none, name path=C] coordinates {(1.5, 0) (1.5, 5.9)};
     \addplot [gray, mark=none, name path=D] coordinates {(1.5, 0) (4.4,-2.9)};
     \addplot [purple, mark=none, name path=E] coordinates {(0,3) (1.5,0)};
     \node[left] at (0,3) {$M$};

    \end{axis}
  \end{tikzpicture}

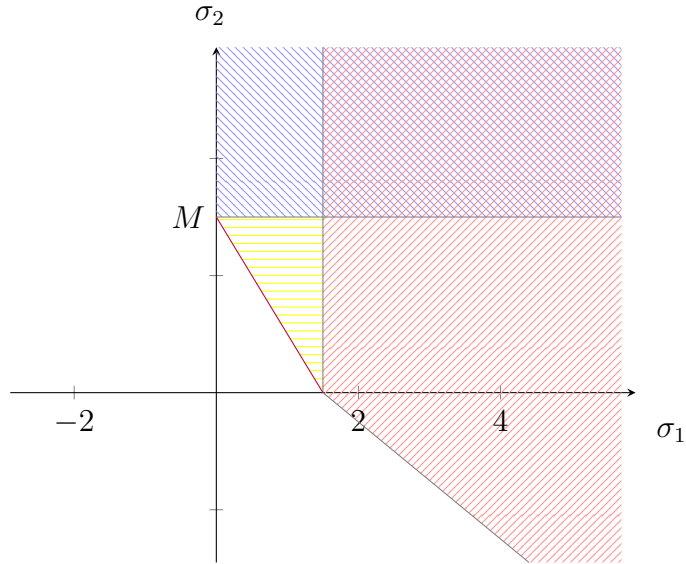
\captionof{figure}{Domain of meromorphy}
\end{center}
Next, we transform this domain using the transformation
$$(s_1,s_2)\mapsto \left(1-s_1,s_1+s_2-\frac12\right)$$
to obtain the following figure:

\bigskip

\begin{center}
\begin{tikzpicture}[scale=1.2]
    \begin{axis}[
        xmin=-2.9, xmax=5.9,
        ymin=-2.9, ymax=5.9,
        axis lines=middle,
        yticklabel=\empty,
        xlabel=$\sigma_1$,ylabel=$\sigma_2$,label style =
               {at={(ticklabel cs:1.1)}}
      ]  
      \addplot[gray, samples=100, domain=0:5.7, name path=A] {3}; 
      \path[name path=xaxis] (\pgfkeysvalueof{/pgfplots/xmin}, 0) -- (\pgfkeysvalueof{/pgfplots/xmax},0);
      \path[name path=max] (\pgfkeysvalueof{/pgfplots/xmin},\pgfkeysvalueof{/pgfplots/ymax} ) -- (\pgfkeysvalueof{/pgfplots/xmax},\pgfkeysvalueof{/pgfplots/ymax});
      \path[name path=second] (1.5,0) -- (4.4,-2.9) -- (5.9,-2.9);
 \path[name path=hull] (0,3) -- (1.5,0);

      \addplot[pattern color=gray!40, pattern=north west lines] fill between[of=A and max, soft clip={domain=0:5.7}];
      \addplot[pattern color=gray!40, pattern=north east lines] fill between[of=max and second, soft clip={domain=1.5:4.4}];
      \addplot[pattern color=gray!40, pattern=north east lines] fill between[of=max and second, soft clip={domain=4.4:5.7}];
      \addplot[pattern color=gray!40, pattern=horizontal lines] fill between[of=A and hull, soft clip={domain=0:1.5}];

     \addplot [gray, mark=none, name path=C] coordinates {(1.5, 0) (1.5, 5.9)};
     \addplot [gray, mark=none, name path=D] coordinates {(1.5, 0) (4.4,-2.9)};
     \addplot [gray, mark=none, name path=E] coordinates {(0,3) (1.5,0)};
     \addplot [red, mark=none, name path=F] coordinates {(1,5.9) (1,2.5)};
     \addplot [red, mark=none, name path=G] coordinates { (-0.5,1) (1,2.5)};
     \addplot [red, mark=none, name path=H] coordinates {(-2.9,1) (-0.5,1)};
    \addplot[pattern color=red!60, pattern=north east lines] fill between[of=max and G, soft clip={domain=-.5:1}];
    \addplot[pattern color=red!60, pattern=north east lines] fill between[of=max and H, soft clip={domain=-2.9:-.5}];
     \node[left] at (0,3) {$M$};
    \end{axis}
  \end{tikzpicture}
  
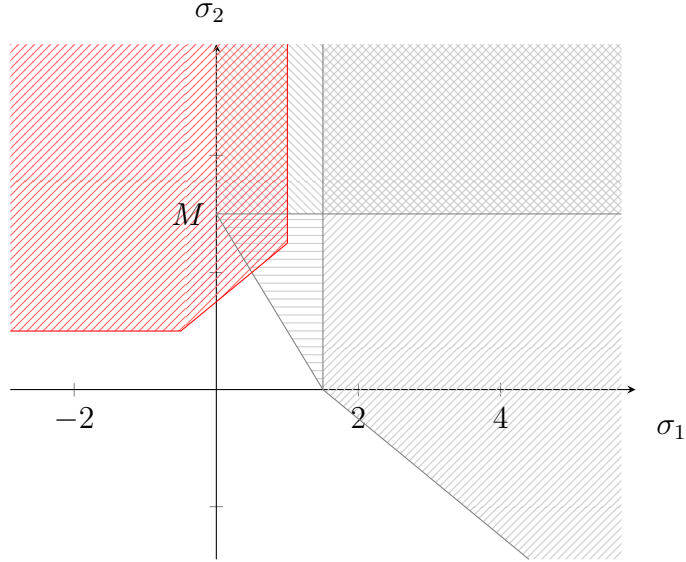
\captionof{figure}{After applying the functional equation}
\end{center}
Now, we can use the properties of tube domains \cite[Theorem 2.5.10]{MR1045639} to conclude that  our $\tilde Z_\chi(s_1,s_2)$ can be extended to a function on the convex hull, which is the whole $\mathbb{C}^2$. 

The fact that $Z(s_1,s_2)$ is a finite linear combination of $\tilde Z_\chi(s_1,s_2)$ concludes the proof of Theorem \ref{thm.main}.
  \bibliographystyle{plain}
 \bibliography{paper}

\begin{thebibliography}{10}

\bibitem{MR1625181}
Bruce~C. Berndt, Ronald~J. Evans, and Kenneth~S. Williams.
\newblock {\em Gauss and {J}acobi sums}.
\newblock Canadian Mathematical Society Series of Monographs and Advanced
  Texts. John Wiley \& Sons, Inc., New York, 1998.
\newblock A Wiley-Interscience Publication.

\bibitem{MR2885075}
Gautam Chinta and Omer Offen.
\newblock Orthogonal period of a {$GL_3(\Bbb Z)$} {E}isenstein series.
\newblock In {\em Representation theory, complex analysis, and integral
  geometry}, pages 41--59. Birkh\"{a}user/Springer, New York, 2012.

\bibitem{MR1693411}
John~E. Cremona.
\newblock Reduction of binary cubic and quartic forms.
\newblock {\em LMS J. Comput. Math.}, 2:64--94, 1999.

\bibitem{MR2041614}
Adrian Diaconu, Dorian Goldfeld, and Jeffrey Hoffstein.
\newblock Multiple {D}irichlet series and moments of zeta and {$L$}-functions.
\newblock {\em Compositio Math.}, 139(3):297--360, 2003.

\bibitem{MR1343325}
Solomon Friedberg and Jeffrey Hoffstein.
\newblock Nonvanishing theorems for automorphic {$L$}-functions on
  {${\operatorname{ GL}}(2)$}.
\newblock {\em Ann. of Math. (2)}, 142(2):385--423, 1995.

\bibitem{MR2254662}
Dorian Goldfeld.
\newblock {\em Automorphic forms and {$L$}-functions for the group
  {$\operatorname{GL}(n,\mathbb{ R})$}}, volume~99 of {\em Cambridge Studies in
  Advanced Mathematics}.
\newblock Cambridge University Press, Cambridge, 2006.
\newblock With an appendix by Kevin A. Broughan.

\bibitem{MR788407}
Dorian Goldfeld and Jeffrey Hoffstein.
\newblock Eisenstein series of {${1\over 2}$}-integral weight and the mean
  value of real {D}irichlet {$L$}-series.
\newblock {\em Invent. Math.}, 80(2):185--208, 1985.

\bibitem{MR1224051}
Jeff Hoffstein.
\newblock Eisenstein series and theta functions on the metaplectic group.
\newblock In {\em Theta functions: from the classical to the modern}, volume~1
  of {\em CRM Proc. Lecture Notes}, pages 65--104. Amer. Math. Soc.,
  Providence, RI, 1993.

\bibitem{MR1045639}
Lars H\"{o}rmander.
\newblock {\em An introduction to complex analysis in several variables},
  volume~7 of {\em North-Holland Mathematical Library}.
\newblock North-Holland Publishing Co., Amsterdam, third edition, 1990.

\bibitem{Lim}
Li-Mei Lim.
\newblock {\em Multiple Dirichlet Series Associated to Prehomogeneous Vector
  Spaces and the Relation with $\operatorname{GL}_3(\mathbb{Z})$ Eisenstein
  Series}.
\newblock PhD thesis, Brown University, Providence, Rhode Island, May 2013.
\newblock Li-Mei Lim's PhD Thesis.

\bibitem{MR676121}
Fumihiro Sat\={o}.
\newblock Zeta functions in several variables associated with prehomogeneous
  vector spaces. {I}. {F}unctional equations.
\newblock {\em T\^{o}hoku Math. J. (2)}, 34(3):437--483, 1982.

\bibitem{MR662121}
Fumihiro Sat\={o}.
\newblock Zeta functions in several variables associated with prehomogeneous
  vector spaces. {III}. {E}isenstein series for indefinite quadratic forms.
\newblock {\em Ann. of Math. (2)}, 116(1):177--212, 1982.

\bibitem{MR695661}
Fumihiro Sat\={o}.
\newblock Zeta functions in several variables associated with prehomogeneous
  vector spaces. {II}. {A} convergence criterion.
\newblock {\em T\^{o}hoku Math. J. (2)}, 35(1):77--99, 1983.

\bibitem{MR0344230}
Mikio Sato and Takuro Shintani.
\newblock On zeta functions associated with prehomogeneous vector spaces.
\newblock {\em Ann. of Math. (2)}, 100:131--170, 1974.

\bibitem{MR289428}
Takuro Shintani.
\newblock On {D}irichlet series whose coefficients are class numbers of
  integral binary cubic forms.
\newblock {\em J. Math. Soc. Japan}, 24:132--188, 1972.

\end{thebibliography}

\end{document}